\documentclass[12pt]{amsart}
\usepackage{graphicx,  amssymb,amsthm, amsfonts, latexsym, epsfig, amsmath}

\newtheorem{definition}{Definition}[section]
\newtheorem{theorem}[definition]{Theorem}
\newtheorem{lemma}[definition]{Lemma}

\newtheorem{corollary}[definition]{Corollary}

\newcommand{\nempty}{\not=\emptyset}

\newcommand{\B}[1]{\ensuremath{\mathbb{#1}}}
\newcommand{\N}{\B{N}}
\newcommand{\Z}{\B{Z}}
\newcommand{\R}{\B{R}}
\newcommand{\C}[1]{\mathcal{#1}}

\newcommand{\ilim}[1]{\ensuremath{\lim\limits_{\leftarrow}\{[0,1],
#1\}}}

\newcommand{\nin}{\not\in}

\renewcommand{\include}{\input}

\newcommand{\w}{\omega}

\renewcommand{\include}{\input}

\newcommand{\sh}{\sigma}

\newcommand{\eps}{\epsilon}

\newcommand{\ba}[1]{\bar{#1}}

\newcommand{\bn}[2]{\bar{#1}^{#2}}

\begin{document}

\title[Inverse limit isotopies]{Homeomorphisms of unimodal inverse limit spaces with a non-recurrent postcritical point}

\author[L. Block]{Louis Block}
\address{Department of Mathematics, University of Florida}

\author[J. Keesling]{James Keesling}
\address{Department of Mathematics, University of Florida}

\author[B.E. Raines]{Brian Raines}
\address{Department of Mathematics, Baylor University, Waco, TX
76798--7328,USA} \email{brian\_raines@baylor.edu}

\author[S. \v Stimac]{Sonja \v Stimac}
\address{ Graduate School of Economics and Business, University of Zagreb, Kennedyev trg 6, 10000 Zagreb, Croatia}
\email{sonja@math.hr}

\subjclass[2000]{37B45, 37E05, 54F15, 54H20} \keywords{attractor,
invariant set, inverse limits, unimodal, continuum, indecomposable}
\maketitle

\begin{abstract}
In this paper we show that the group of automorphisms of a
non-recurrent tent map inverse limit is very simple by demonstrating that every
homeomorphism of such a space is isotopic to a power of the induced
shift homeomorphism. \textbf{Note:  This paper appeared in \emph{Topology and its Applications} 156 (2009), no. 15, 2417-2425}
\end{abstract}

\section{Introduction}

In the last fifteen years inverse limits of unimodal maps have been
studied extensively.  One of the main problems in the field of study
is to classify all such spaces based upon the dynamics of the
particular unimodal map that generates the inverse limit space.
There are many known topological invariants in this class of spaces
such as endpoints, \cite{Barge-Martin-1995}, \cite{Bruin-1999},
folding points, \cite{Bruin-2000}, \cite{Good-Knight-Raines-2006},
\cite{Raines-2004}, asymptotic arc components, \cite{Bruin-2005},
and complicated subcontinua, \cite{Barge-et-al-1996},
\cite{Brucks-Bruin-1999}, \cite{Bruin-2007}. The main conjecture is
due to W.T. Ingram:
\begin{center} \textbf{Ingram's Conjecture:}
\end{center}
\begin{quote}
Let $T_s$ and $T_t$ be tent maps with slopes $s$ and $t$
respectively.  Then $\ilim{T_s}$ is homeomorphic with $\ilim{T_t}$
if and only if $s=t$
\end{quote}

Ingram's conjecture has been proved in many special cases.  If $T_s$
is a tent map with a periodic critical point of period $n$ and $T_t$
is a tent map with a periodic critical point of period $n'$ then
Barge and Martin proved that $\ilim{T_s}$ has $n$ endpoints and
$\ilim{T_t}$ has $n'$ endpoints, \cite{Barge-Martin-1995}.  Hence if
$n\neq n'$ then $\ilim{T_s}$ is not homeomorphic with $\ilim{T_t}$.
Bruin extended this by introducing the notion of folding points, and
showing that if $T_s$ and $T_t$ have preperiodic critical points of
order $n$ and $n'$ respectively then $\ilim{T_s}$ and $\ilim{T_t}$
have $n$ and $n'$ many folding points respectively,
\cite{Bruin-2000}. Hence if $n\neq n'$ then the associated inverse
limit spaces are not homeomorphic. Barge and Diamond proved the
conjecture in the case that $T_s$ is one of the three tent maps with
a periodic critical point of period 5, \cite{Barge-Diamond-1995}.
In a series of papers Kailhofer proved Ingram's conjecture in the
case that the periodic point is periodic, \cite{Kailhofer-2002},
\cite{Kailhofer-2003} (see \cite{blockjagkailkees} which Kailhofer
wrote with Block, Jakimovik, and Keesling for a particularly
readable account of her proof.) Independently \v{S}timac proved in
her dissertation research that Ingram's conjecture holds in the case
that the critical point is periodic, \cite{Stimac-thesis-2002}, and
then she extended her work to the case that the critical point is
preperiodic, \cite{Stimac-2007}. Recently Raines and \v{S}timac have
proved the Ingram conjecture in the case that the critical point is
non-recurrent, \cite{rainesstimac:nr}.

The focus of this paper is the case that the inverse limit is
induced by a tent map with a non-recurrent critical point.  Here we
consider the structure of the group of automorphisms on such a
space, and we prove that this group is isomorphic to $\Z$ by
showing:\begin{quote}   \textbf{Main Theorem:}  Let $T$ be a tent
map with a non-recurrent critical point.  Let $h:\ilim{T}\to
\ilim{T}$ be a homeomorphism. Then there is an integer $k$ such that
$h$ is isotopic to $\sh^k$.
\end{quote} By $\sh$ we mean the inverse of the natural shift homeomorphism
$$\sh(x_0, x_1\dots)=(T(x_0), x_0, x_1, x_2\dots)$$ on $\ilim{T}$.

This extends recent work by Block, Jagimovik, Kailhofer and Keesling
who proved this result in the case that the critical point is
periodic, \cite{blockjagkailkees}, and in this paper we adopt many
of their techniques.  The main difference between the periodic case
and the non-recurrent case is that in the periodic case there are
only finitely many folding points (points $\ba x\in \ilim{T}$ with
the property that there is no neighborhood of $\ba x$ homeomorphic
to the product of a zero-dimensional set and an arc), and all of
these folding points are in fact endpoints.  In the case we
consider, the non-recurrent case, there are no endpoints, but we
have (perhaps uncountably many) folding points.  These folding
points present the main difficulty in proving our result.

\section{Definitions and Preliminary Lemmas}

Let $T=T_s:[0,1]\to [0,1]$ be a tent-map with slope, $s$, between
$\sqrt{2}$ and $2$ such that the critical point of $T$, $1/2$, is
non-recurrent.  Let $C_0$ denote the composant of the single
endpoint, $(0,0,0,\dots )$, of $\ilim{T}=K$.  Suppose that
$$h:K\to K$$ is a homeomorphism.  Then we have $h(C_0)=C_0$.

Let $p\in \Z_+$.  Define the point $\ba x\in K$ to be a
\emph{$p$-point} if $\pi_n(\ba x)=1/2$ for some $n>p$.  Let $E_p$ be
the set of all $p$-points in $K$.

Let $S\in \N$ be large enough to satisfy the conditions from
\cite{rainesstimac:nr}.  These conditions are quite technical and
will mostly not be important in this paper.  A few of the
implications, however, of $n\ge S$ will be important, and we mention
them below. We write $\C C\prec \C D$ if the chaining $\C C$ refines
the chaining $\C D$, and we define the \emph{mesh} of $\C C$ to be
the largest diameter of any of its links. In \cite{rainesstimac:nr},
we construct a sequence of chainings of $K$, $\{\C C_{k,r}\}_{k\in
\Z_+, r\ge S}$.  Let $k\in \Z_+$ and $r\ge S$.  Let $\C P$ be the
partition of $[0, T(1/2)]$ induced by the collection of points
$$\bigcup_{j=0}^{k+r+1} T^{-j}(1/2)$$  Let $t=|\C P|$, and suppose that $\C P=\{x_1<x_2\cdots <x_t\}$.  Define the
following open cover of $[0, T(1/2)]$, $$I^1_{k,r}=[0, x_2)$$
$$I^j_{k,r}=(x_{j-1}, x_{j+1})$$ for $1<j<t$ and
$$I^t_{k,r}=(x_{t-1}, T(1/2)]$$  Let
$$\ell^j_{k,r}=\pi_k^{-1}(I^j_{k,r})$$ for $1\le j\le t$.  Then it
is not hard to see that $$\C C_{k,r}=\{\ell^j_{k,r}|1\le j\le t\}$$
is a chaining of $K$ with open sets.

We show in \cite{rainesstimac:nr} that this collection of chains
satisfies: \begin{enumerate}\item $\C C_{q,m}\prec \C C_{p,n}$
provided $q\ge p$ and $m\ge n$ (\cite[Lemma 2.7]{rainesstimac:nr});
\item the mesh of $\C C_{q,m}$ goes to zero as $q,m\to \infty$;
\item each $p$-point, $\ba x$, is contained in a a
link of $\C C_{p,n}$, $\ell^x_{p,n}$, such that if $A$ is the arc
component of $\ell^x_{p,n}$ that contains $\ba x$ then $A\cap
E_p=\{\ba x\}$(\cite[Lemma 2.8]{rainesstimac:nr}).\end{enumerate}
Let $p$ and $n$ be given (with $n\ge S$), and choose $q> p$ and $m>
n$ such that
$$h(\C C_{q,m})\prec \C C_{p,n}$$ Let $\ba x\in E_q\cap C_0$.  We
showed that if $A$ is the arc component of a link of $\C C_{p,n}$
which contains $h(\ba x)$ then there is a unique $p$-point, $\ba
z$, in $A$ (\cite[Lemma 3.1]{rainesstimac:nr}.  We defined an
`adjusted' map, $h_{q,p}$, which maps $\ba x$ to this $p$-point,
$\ba z$. Then we extended $h_{q,p}$ in a natural monotonic way on
the arcs between adjacent $q-$points in $C_0$.

We showed in \cite{rainesstimac:nr} that there is a $q,p\in \Z_+$
and $m,n\ge S$ such that $$h(\C C_{q,m})\prec \C C_{p,n}$$ and for
$b=p-q$ we have that $$h_{q,p}|_{E_q\cap C_0}=\sh^{-b}|_{E_q\cap
C_0}$$

Let $F\subseteq K$ be the set of folding points for $K$, i.e. $\ba
x\in F$ if, and only if $\pi_n(\ba x)=x_n\in \w(1/2)$ for all $n\in
\N$ (equivalently $\ba x$ is a limit point of a sequence $\bn{y}{n}$
such that $\bn{y}{n}\in E_n$). In \cite{Raines-2004} we show that
$\ba x\in F$ if and only if every neighborhood of $\ba x$ is not
homeomorphic to the product of a zero-dimensional set and an open
arc.

\begin{lemma}  Let $\ba x\in F$ then $h(\ba x)=\sh^{-b}(\ba x)$.
\end{lemma}

\begin{proof}
 Let $\ba x\in F$, and let $\ba y=h(\ba x)$.  Since $h$ is a homeomorphism, $\ba y\in F$.  Let $p_1, q_1\in
 \Z_+$ and $m_1, n_1\ge S$ such that $$h(\C C_{q_1, m_1})\prec \C
 C_{p_1, n_1}\prec h(\C C_{q, m})\prec \C C_{q,n}$$ and recursively
 define $p_j, q_j\in \Z_+$ and $m_j, n_j\ge S$ such that $$h(\C C_{q_j, m_j})\prec \C
 C_{p_j, n_j}\prec h(\C C_{q_{j-1}, m_{j-1}})\prec \C
 C_{p_{j-1},n_{j-1}}$$  For each $j\in \N$, let $\ell^x_{q_j, m_j}$
 be a link of $\C C_{q_j,m_j}$ which contains $\ba x$ and let
 $\ell^y_{p_j, n_j}$ be a link of $\C C_{p_j, n_j}$ which contains
 $h(\ell^x_{q_j,m_j})$.  Define $\bn{z}{j}\in E_{q_j}\cap C_0$ such that $\bn{z}{j}\in
 \ell^x_{q_j, m_j}$.  Then we must have:\begin{enumerate}
   \item $\bn{z}{j}\to \ba x$ as $j\to \infty$;
   \item $h(\bn{z}{j})\in \ell^y_{p_j,n_j}$ and hence
   $\sh^{-b}(\bn{z}{j})=h_{q_j,p_j}(\bn{z}{j})\in \ell^y_{p_j, n_j}$;
   \item since the mesh of $\C C_{p_j,n_j}$ goes to zero, $\sh^{-b}(\bn{z}{j})=h_{q_j,p_j}(\bn{z}{j})\to \ba y$.
 \end{enumerate}  Thus $$h(\ba x)=\sh^{-b}(\ba x)$$
\end{proof}

Following \cite{blockjagkailkees}, we define the following
`composant-metric' $\ba d$.  Let $C$ be some composant of $K$ and
let $\ba x, \ba y\in C$.  Choose $n\in \N$ to be large enough such
that if $A$ is the arc with endpoints $\ba x$ and $\ba y$ in $C$
then $\pi_n|_{A}$ is a homeomorphism.  Then define $$\ba d(\ba x,
\ba y)=s^n|\pi_n(\ba x)-\pi_n(\ba y)|$$ (recall that $s$ is the
slope of the tent map we are using as the bonding map for $K$).
Notice that for all $m\ge n$ we have $$\ba d(\ba x, \ba
y)=s^m|\pi_m(\ba x)-\pi_m(\ba y)|$$

Let $\{D_i\}_{i\in \N}$ be a sequence of compact sets in some
compact metric space $Y$.  Define $$\limsup\{D_i\}=\{y\in Y|\text{
for some subsequence }\{D_{i_j}\}\text{ and }y_{i_j}\in D_{i_j},
y_{i_j}\to y\}$$

The next three lemmas are based upon lemmas from
\cite{blockjagkailkees} and are used throughout this paper.  The
proofs are virtually identical. We only change `end point' to
`folding point.'

\begin{lemma}\cite[5.1]{blockjagkailkees}\label{lemma5.1}
    Suppose that $A$ is an arc in $K$ not containing a folding point
    of $K$.  Then there is a neighborhood $V$ of $A$ homeomorphic to
    $C\times I$ where $C$ is a zero-dimensional set.  The boundary
    of $V$ corresponds to $C\times \{0,1\}$.  Moreover there is a
    positive integer $m$ such that $\pi_m|_B$ is a homeomorphism for
    each arc $B$ in $V$, and each component of
    $V$ has the same $\ba d$ length.
\end{lemma}

Let $\ba l$ denote the length of an arc under the metric $\ba d$.

\begin{lemma}\cite[5.4]{blockjagkailkees}\label{lemma5.4}
  Let $\{A_i\}_{i\in \N}$ be a sequence of arcs in $K$.  Suppose
  that $A_i\to B$ in the Hausdorff metric.  Suppose also that there
  is an $M>0$ such that $\ba l(A_i)\le M$ for all $i$.  Then $B$ is
  an arc in $K$ and $\ba l(B)\le M$.
\end{lemma}

\begin{lemma}\cite[5.5]{blockjagkailkees}\label{lemma5.5}
  Let $\{A_i\}$ be a sequence of arcs in $K$ with endpoints
  $\bn{a}{i}$ and $\bn{b}{i}$ respectively.  Suppose that there is a
  positive number $M$ such that $\ba d(\bn{a}{i}, \bn{b}{i})\le M$
  for each $i\in \N$.  Suppose also that $\bn{a}{i}$ converges to
  some $\ba a\in K$.  Then $B=\limsup\{A_i\}$ is an arc in $K$ and
  $\ba l(B)\le 2M$.
\end{lemma}

\begin{lemma}
  Let $C$ be a composant of $K$, and let $\ba z\in C$, then $h(\ba z)\in
  \sh^{-b}(C)$.
\end{lemma}

\begin{proof}
  Let $\ba x\in C\cap E_q$.  Then there is a sequence of points $\bn{x}{n}\in C_0\cap  E_q$
  such that $\bn{x}{n}\to \ba x$.  Since $h_{q,p}(\bn{x}{n})=\sh^{-b}(\bn{x}{n})$ we see that
  $$\ba d(h(\bn{x}{n}), \sh^{-b}(\bn{x}{n}))<2\epsilon\cdot s^q$$ where $\epsilon>0$ is the mesh of the chaining $\C C_{p,n}$.  So by Lemma \ref{lemma5.5},
  if we let $A_n$ be the arc in $C_0$ with endpoints $\sh^{-b}(\bn{x}{n})$ and
  $h(\bn{x}{n})$ then we see that $\ba l(A_n)<2\epsilon\cdot s^q$ and hence $B=\limsup\{A_n\}$ is an arc in $K$.  This arc contains $\sh^{-b}(\ba x)$ and
  $h(\ba x)$.  Thus $h(\ba x)\in \sh^{-b}(C)$.  Since $h$ is a homeomorphism,
  this shows that $h(\ba z)\in \sh^{-b}(C)$.
\end{proof}

\begin{lemma}
  Let $C$ be a composant of $K$.  Let $\bn{x}{i}, \bn{x}{i+1}\in C$
  be adjacent in $E_p$, i.e., if $A$ is the arc with
  endpoints $\bn{x}{i}$ and $\bn{x}{i+1}$ then $A$ contains no other $p$-points.
  Then
  $\ba{d}(\bn{x}{i}, \bn{x}{i+1})\le s^p$ where $s$ is the slope of $T$.
\end{lemma}

\begin{proof}
  Let $A$ be the arc in $C$ with endpoints $\bn{x}{i}$ and $\bn{x}{i+1}$.  Since $\bn{x}{i}$ and $\bn{x}{i+1}$ are adjacent $p$-points we
  know that $\pi_p|_A$ is a homeomorphism.  Thus $$\ba d(\bn{x}{i},
  \bn{x}{i+1})=s^p|\pi_p(\bn{x}{i})-\pi_p(\bn{x}{i+1})|\le s^p$$
\end{proof}

Recall that each composant, $C\neq C_0$, of $K$ is the continuous
image of $\R$ under a continuous one-to-one function $g_C$.  Whereas
$C_0$ is homeomorphic to $\R^+$ via a homeomorphism $g$.   Let $<$
be defined on $C\neq C_0$ by $\ba x<\ba y$ if and only if
$g_C^{-1}(\ba x)<g_C^{-1} (\ba y)$, and define $<$ on $C_0$ in a
similar way.

\begin{lemma}
  Let $C$ be a composant of $K$ and let $\ba x<\ba y$ in $C$.  Then
  $h(\ba x)<h(\ba y)$ if and only if $\sh^{-b}(\ba x)<\sh^{-b}(\ba y)$.
\end{lemma}

\begin{proof}
  Since $h$ is a homeomorphism of $C$ it is either order-preserving or reversing.  Let $\bn{x}{i}$
  and $\bn{x}{i+1}$ be adjacent $q$-points in $C$ with
  $\bn{x}{i}<\bn{x}{i+1}$.  By \cite{rainesstimac:nr} we know that
  $h(\bn{x}{i})$ and $\sh^{-b}(\bn{x}{i})$ are on the same arc component of a
  link of $\C C_{p,n}$, $A_i$.  Also $h(\bn{x}{i+1})$ and
  $\sh^{-b}(\bn{x}{i+1})$ are on the same arc component of a link of $\C C_{p,n}$, $A_{i+1}$.  It is clear
  that every point of $A_i$ is less than every point of $A_{i+1}$ or every point of $A_i$ is greater than every point of $A_{i+1}$ depending upon whether $\sh^{-b}$ is order-preserving or reversing.  Hence $h(\bn{x}{i})<h(\bn{x}{i+1})$ if and only if $\sh^{-b}(\ba x)<\sh^{-b}(\ba y)$.
\end{proof}

\begin{lemma}
  There is a real number $M>0$ such that $\ba d(\sh^{-b}(\ba z), h(\ba
  z))\le M$.
\end{lemma}

\begin{proof}
  Let $\eps>0$ be the mesh of $\C C_{p,n}$.  Then notice that the $\overline{d}$-length of an arc component of a link of $\C C_{p,n}$
  is at most $2\epsilon\cdot s^p$ since these arc components contain at most one $p$-point.  Let $\bn{x}{i}$ and
  $\bn{x}{i+1}$ be adjacent $q$-points in $C$ with $\bn{x}{i}\le  \ba
  z<\bn{x}{i+1}$.  Then, without loss of generality, $h(\bn{x}{i})\le h(\ba z)<h(\bn{x}{i+1})$ and $\sh^{-b}(\bn{x}{i})\le \sh^{-b}(\ba z)<\sh^{-b}(\bn{x}{i+1})$.  By
  \cite{rainesstimac:nr}, $\sh^{-b}(\bn{x}{i})$ and $h(\bn{x}{i})$ are on the
  same arc-component of a link of $\C C_{p,n}$, and the same is true
  for $\sh^{-b}(\bn{x}{i+1})$ and $h(\bn{x}{i+1})$.  Let $$\ba
  a=\min\{\sh^{-b}(\bn{x}{i}), h(\bn{x}{i})\}$$ and let $$\ba
  b=\max\{\sh^{-1}(\bn{x}{i+1}), h(\bn{x}{i+1})\}$$  Let $B$ be the arc with endpoints $\ba a$ and $\ba b$.
  We see that both $\sh^{-b}(\ba z)$ and $h(\ba z)$ are in $B$.  Moreover the
  length of $B$ is less than or equal to the length of the arc from
  $\sh^{-b}(\bn{x}{i})$ to $\sh^{-b}(\bn{x}{i+1})$ plus the lengths of the arc
  components of a link of $\C C_{p,n}$ which contains a $p$-point.
  That is to say $$\ba d(\sh^{-b}(\ba z), h(\ba z))\le \ba l(B)\le s^q+4\eps
  s^q$$
\end{proof}

\section{Isotopy}

\begin{lemma}
  Let $\ba x\in K\setminus F$, and suppose that $(\bn{x}{n})_{n\in
  \N}\subseteq K$ be a sequence such that $\bn{x}{n}\to \ba x$.
  Then the arcs $A_n$ with endpoints $\bn{x}{n}$ and $h(\bn{x}{n})$
  converge in the Hausdorff metric to the arc $A$ with endpoints
  $\ba x$ and $h(\ba x)$ if there is an $N\in \N$ such
  that for all $n\ge N$, $A_n\cap E_m=\emptyset$, (where $m$ is
  chosen as in Lemma \ref{lemma5.1} for the arc $A$).
\end{lemma}

\begin{proof}
  Notice that since $\ba x\nin F$, the arc $A$ with endpoints $\ba x$
  and $h(\ba x)$ does not intersect $F$.  To see this, suppose that
  there is a folding point, $\ba z$, between $\ba x$ and $h(\ba x)$.
  then we would have, say, $\ba x<\ba z=h(\ba z)<h(\ba x)$ which would
  imply that $h$ is order-reversing on $C$, a contradiction.  So $A$
  contains no folding points of $K$ and we can apply Lemma
  \ref{lemma5.1}

  Let $m\in \N$ be chosen as in Lemma \ref{lemma5.1}, and let $V$ be the
  neighborhood of $A$ as described in that lemma.  Let $N\in \N$ be
  large enough so that $A_n\cap E_m=\emptyset$.  Then $\pi_m|_{A_n}$
  is a homeomorphism for all $n\ge N$.  Let $N'\ge N$ be defined so
  that for all $n\ge N'$, $\bn{x}{n}$ and $h(\bn{x}{n})$ are in $V$.
  Then for all $n\ge N'$, $A_n\subseteq V$.  It follows that $A_n\to
  A$ in the Hausdorff metric.
\end{proof}

Let $\ba z\in C$ a composant of $K$.  Let $\delta>0$.  By $\ba
z+\delta$ we mean the point $\ba y\in C$ such that $\ba z<\ba y$ and
such that $\overline{d}(\ba z, \ba y)=\delta$.  We define $\ba
z-\delta$ in a similar fashion.

\begin{lemma}\label{znotfoldpoint}
  Let $\ba z\in K\setminus F$.  If
  $(\bn{z}{n})_{n\in \N}\subseteq K$ with $\bn{z}{n}\to \ba z$, then
  the arcs, $A_n$, with endpoints $\bn{z}{n}$ and $h(\bn{z}{n})$
  converge in the Hausdorff metric to $A$ the arc with endpoints
  $\ba z$ and $h(\ba z)$.
\end{lemma}

\begin{proof}
  If the composant, $C$, which contains $\ba z$ does not
  contain any folding points, then since $\limsup \{A_n\}$
  is a subarc of $C$, Lemma \ref{lemma5.5}, we have that $\limsup\{ A_n\}$ does not contain
  any folding points.  Hence the
  sequence of arcs, $\{A_n\}$, satisfies the previous lemma.

  Now suppose that $C$ contains a folding point, $\ba x$.  Without
  loss of generality, suppose that $\ba z<\ba x$, that $\limsup
  \{A_n\}\ni \ba x$, and that $\ba z\le h(\ba z)<\ba x$.  Let $m\in \N$
  and $V$ be chosen as in Lemma \ref{lemma5.1} for the arc, $A$, with endpoints $\ba
  z$ and $h(\ba z)$ (notice that if $\ba z=h(\ba z)$ then $A$ is
  a `degenerate arc.')  Let $(n_j)_{j\in\N}$ be defined so that
  there is a sequence $(\bn{t}{n_j})_{j\in \N}$, $\bn{t}{n_j}\in
  A_{n_j}$, with $\bn{t}{n_j}\to \ba x$.  We lose no generality in
  assuming that $\bn{t}{n_j}\in E_m$ for all $j\in \N$.  Let $J\in
  \N$ be large enough so that $\bn{z}{n_j}, h(\bn{z}{n_j})\in V$ for
  all $j\ge J$.  See Figure \ref{fig:first}.  We also assume that $\bn{z}{n_j}<h(\bn{z}{n_j})$. The
  case that $h(\bn{z}{n_j})<\bn{z}{n_j}$ is handled similarly.

\begin{figure}
\unitlength=10mm
\begin{picture}(12,8)(0,0)


\put(1,1){\line(1,0){10}}

\put(2,1){\circle*{0.15}} \put(7.5,1){\circle*{0.15}}
\put(2,.5){\small $\ba z$}

\put(2, 3){\vector(0,-1){1.5}}

\put(7.5,.5){\small $h(\ba z)$} \put(7.5, 2.5){\vector(0,-1){1}}

 \put(9.5,.5){\small $\ba x=h(\ba x)$}

\put(9.5, 3.25){\vector(0,-1){1.5}}

\put(9.5,1){\circle*{0.15}}\put(5,1.25){\small $A$}


\put(1,6){\line(1,0){7.5}} \put(8.5,5.75){\oval(.5,.5)[r]}

\put(5,6.25){\small $A_{n_j}$} \put(1,5.5){\line(1,0){7.5}}

\put(2.25,6){\circle*{0.15}}

\put(2.25, 6.25){\small $\bn{z}{n_j}$}

\put(7.15,5.5){\circle*{0.15}}

\put(7.15, 5){\small $h(\bn{z}{n_j})$}

\put(8.75,5.75){\circle*{0.15}}

\put(9, 5.75){\small $\bn{t}{n_j}$}


\put(1,4){\line(1,0){8}} \put(9,3.75){\oval(.5,.5)[r]}

\put(5,4.25){\small $A_{n_{j+1}}$}

\put(1,3.5){\line(1,0){8}}

\put(2,4){\circle*{0.15}}

\put(1.75, 4.25){\small $\bn{z}{n_{j+1}}$}

\put(7.25,3.5){\circle*{0.15}}

\put(6.8, 3){\small $h(\bn{z}{n_{j+1}})$}

\put(9.25,3.75){\circle*{0.15}}

\put(9.5, 3.75){\small $\bn{t}{n_{j+1}}$}


\put(1.25,.25){\dashbox{.1}(7.15,6.5)}

\put(4.5 ,6.9){\small $V$}

\end{picture}
\caption{Arrangement of points.} \label{fig:first}
\end{figure}
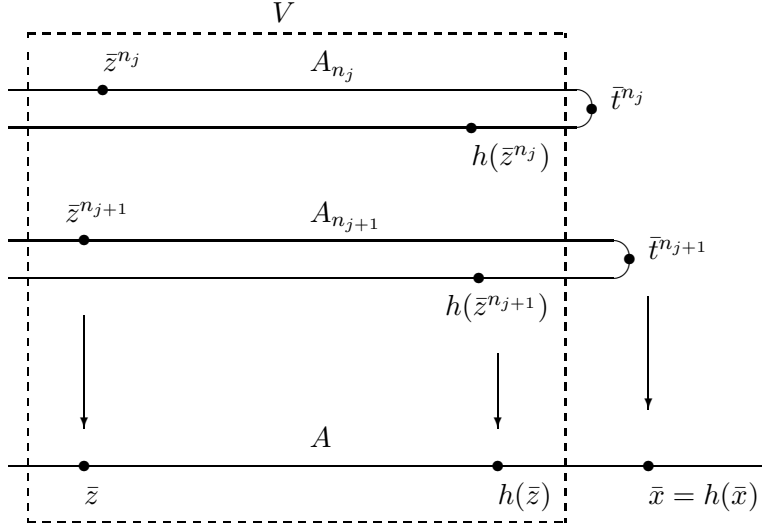

  First suppose that $\ba z\neq h(\ba z)$ so that $A$ is a
  non-degenerate arc.  Let $\delta>0$ be defined so that
  $\ba z<\ba z+\delta<h(\ba z)-\delta<h(\ba z)$.  For each $j\ge J$, 
let $C_{n_j}$ be the arc component of $V$ which
  contains $h(\bn{z}{n_j})$.  


%
%
%
%
%
%
%
%
%
%
%
%
%
%
%
%
%
%
%
%
%
%
%
%
%
%
%
%
%
%
%
%
%



\begin{figure}
\unitlength=10mm
\begin{picture}(12,8)(0,0)


\put(1,1){\line(1,0){10}}

\put(2,1){\circle*{0.15}} \put(2,.5){\small $\ba z$}
\put(3,1){\circle*{0.15}}\put(3,.5){\small $\ba z+\delta$}

\put(3,1.2){\line(0,1){.1}}\put(3,1.4){\line(0,1){.1}}\put(3,1.6){\line(0,1){.1}}\put(3,1.8){\line(0,1){.1}}\put(3,2){\line(0,1){.1}}

\put(3,2.2){\line(0,1){.1}}\put(3,2.4){\line(0,1){.1}}\put(3,2.6){\line(0,1){.1}}\put(3,2.8){\line(0,1){.1}}\put(3,3){\line(0,1){.1}}

\put(3,3.2){\line(0,1){.1}}\put(3,3.4){\line(0,1){.1}}\put(3,3.6){\line(0,1){.1}}\put(3,3.8){\line(0,1){.1}}\put(3,4){\line(0,1){.1}}

\put(3,4.2){\line(0,1){.1}}\put(3,4.4){\line(0,1){.1}}\put(3,4.6){\line(0,1){.1}}\put(3,4.8){\line(0,1){.1}}\put(3,5){\line(0,1){.1}}


\put(7.5,1){\circle*{0.15}} \put(7.5,.5){\small $h(\ba z)$}

\put(6.5,1){\circle*{0.15}}\put(5.75, .5){\small $h(\ba z)-\delta$}


\put(6.5,1.2){\line(0,1){.1}}\put(6.5,1.4){\line(0,1){.1}}\put(6.5,1.6){\line(0,1){.1}}\put(6.5,1.8){\line(0,1){.1}}\put(6.5,2){\line(0,1){.1}}

\put(6.5,2.2){\line(0,1){.1}}\put(6.5,2.4){\line(0,1){.1}}\put(6.5,2.6){\line(0,1){.1}}\put(6.5,2.8){\line(0,1){.1}}\put(6.5,3){\line(0,1){.1}}

\put(6.5,3.2){\line(0,1){.1}}\put(6.5,3.4){\line(0,1){.1}}\put(6.5,3.6){\line(0,1){.1}}\put(6.5,3.8){\line(0,1){.1}}\put(6.5,4){\line(0,1){.1}}

\put(6.5,4.2){\line(0,1){.1}}\put(6.5,4.4){\line(0,1){.1}}\put(6.5,4.6){\line(0,1){.1}}\put(6.5,4.8){\line(0,1){.1}}\put(6.5,5){\line(0,1){.1}}


 \put(9.5,.5){\small $\ba x=h(\ba x)$}

\put(9.5,1){\circle*{0.15}}\put(5,1.25){\small $A$}


\put(1,6){\line(1,0){7.5}}

\put(7,6.2){\small $B_{n_j}$}


\put(8.5,5.5){\oval(1,1)[r]}

\put(1,5){\line(1,0){7.5}} \put(1.75,4.5){\small $C_{n_j}$}

\put(1.25, 4.9){\small $($} \put(8.25, 4.9){\small $)$}


\put(2.25,6){\circle*{0.15}}

\put(2.25, 6.25){\small $\bn{z}{n_j}$}

\put(7.15,5){\circle*{0.15}}

\put(7.15, 4.5){\small $h(\bn{z}{n_j})$}

\put(9,5.5){\circle*{0.15}}

\put(9.25, 5.5){\small $\bn{t}{n_j}$}




\put(3,5){\circle*{0.15}} \put(3.2, 4.5){\small $\bn{s}{n_j}$}

\put(6.5,5){\circle*{0.15}}\put(5.9, 4.5){\small $\bn{u}{n_j}$}


\put(1.25,.25){\dashbox{.05}(7.15,6.5)}

\put(4.5 ,6.9){\small $V$}

\end{picture}
\caption{Reference points on $A_{n_j}$.} \label{fig:third}
\end{figure}
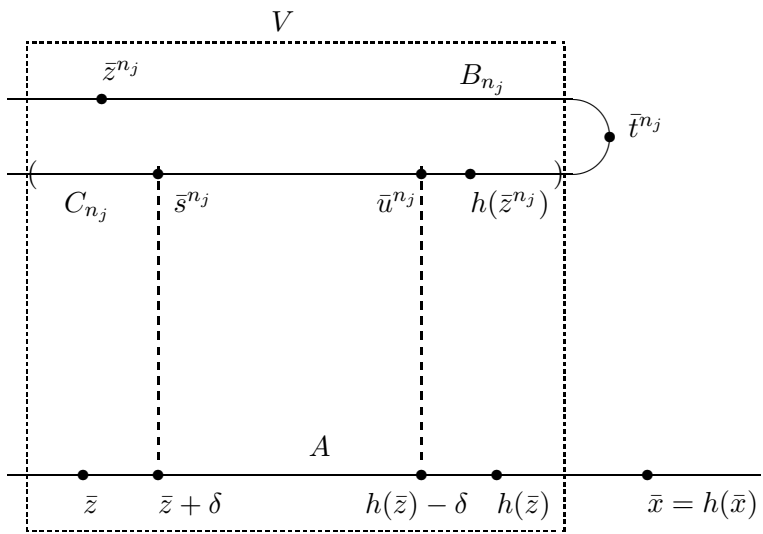

 Define $\bn{s}{n_j}\in C_{n_j}$  such that
  $$\pi_m(\bn{s}{n_j})=\pi_m(\ba{z}+\delta)$$  Also define
  $\bn{u}{n_j}\in C_{n_j}$ such that
  $$\pi_m(\bn{u}{n_j})=\pi_m(h(\ba{z})-\delta)$$

   See Figure \ref{fig:third}.

    Consider
  $h^2(\bn{z}{n_j})$.  Since $\bn{z}{n_j}<h(\bn{z}{n_j})$ and $h$ is
  order-preserving, we have $$\bn{z}{n_j}<h(\bn{z}{n_j})<h^2(\bn{z}{n_j})$$
   Since the arc $A_{n_j}$ has only one $m$-point, we have
    \begin{enumerate}
     \item $\bn{u}{n_j}<h^2(\bn{z}{n_j})<\bn{s}{n_j}$ for infinitely
     many $j\ge J$,
     \item $\bn{s}{n_j}<h^2(\bn{z}{n_j})$ for infinitely many $j\ge
     J$, or
     \item $h^2(\bn{z}{n_j})<\bn{u}{n_j}$ for all $j\ge
     J$.
   \end{enumerate}

  In case (1) we pass to a subsequence of $\bn{z}{n_j}$ if
  necessary.  Then there is some $\ba w\in A$ with
  $h^2(\bn{z}{n_j})\to \ba w$ and $\ba z+\delta\le w\le h(\ba
  z)-\delta$.  Thus $\ba w\nin\{ \ba z, h(\ba z)\}$.  Since $h$ is
  continuous though, $h(\bn{z}{n_j})\to h^{-1}(\ba w)$.  Hence
  $h^{-1}(\ba w)=h(\ba z)$.  This implies that $\ba w=h^2(\ba z)$ and $\ba
  z<h^2(\ba z)<h(\ba z)$, which contradicts the fact that $h$ is
  order-preserving.

\begin{figure}
\unitlength=10mm
\begin{picture}(12,8)(0,0)


\put(1,1){\line(1,0){10}}

\put(2,1){\circle*{0.15}} \put(2,.5){\small $\ba z$}
\put(3,1){\circle*{0.15}}\put(3,.5){\small $\ba z+\delta$}

\put(3,1.2){\line(0,1){.1}}\put(3,1.4){\line(0,1){.1}}\put(3,1.6){\line(0,1){.1}}\put(3,1.8){\line(0,1){.1}}\put(3,2){\line(0,1){.1}}

\put(3,2.2){\line(0,1){.1}}\put(3,2.4){\line(0,1){.1}}\put(3,2.6){\line(0,1){.1}}\put(3,2.8){\line(0,1){.1}}\put(3,3){\line(0,1){.1}}

\put(3,3.2){\line(0,1){.1}}\put(3,3.4){\line(0,1){.1}}\put(3,3.6){\line(0,1){.1}}\put(3,3.8){\line(0,1){.1}}\put(3,4){\line(0,1){.1}}

\put(3,4.2){\line(0,1){.1}}\put(3,4.4){\line(0,1){.1}}\put(3,4.6){\line(0,1){.1}}\put(3,4.8){\line(0,1){.1}}\put(3,5){\line(0,1){.1}}


\put(7.5,1){\circle*{0.15}} \put(7.5,.5){\small $h(\ba z)$}

\put(6.5,1){\circle*{0.15}}\put(5.75, .5){\small $h(\ba z)-\delta$}


\put(6.5,1.2){\line(0,1){.1}}\put(6.5,1.4){\line(0,1){.1}}\put(6.5,1.6){\line(0,1){.1}}\put(6.5,1.8){\line(0,1){.1}}\put(6.5,2){\line(0,1){.1}}

\put(6.5,2.2){\line(0,1){.1}}\put(6.5,2.4){\line(0,1){.1}}\put(6.5,2.6){\line(0,1){.1}}\put(6.5,2.8){\line(0,1){.1}}\put(6.5,3){\line(0,1){.1}}

\put(6.5,3.2){\line(0,1){.1}}\put(6.5,3.4){\line(0,1){.1}}\put(6.5,3.6){\line(0,1){.1}}\put(6.5,3.8){\line(0,1){.1}}\put(6.5,4){\line(0,1){.1}}

\put(6.5,4.2){\line(0,1){.1}}\put(6.5,4.4){\line(0,1){.1}}\put(6.5,4.6){\line(0,1){.1}}\put(6.5,4.8){\line(0,1){.1}}\put(6.5,5){\line(0,1){.1}}


 \put(9.5,.5){\small $\ba x=h(\ba x)$}
\put(4.5, 1){\circle*{0.15}}\put(4.25, .5){\small $\ba w$}

\put(9.5,1){\circle*{0.15}}\put(5,1.25){\small $A$}


\put(1,6){\line(1,0){7.5}}



\put(8.5,5.5){\oval(1,1)[r]}

\put(1,5){\line(1,0){7.5}} \put(1.75,4.5){\small $C_{n_j}$}

\put(1.25, 4.9){\small $($} \put(8.25, 4.9){\small $)$}


\put(2.25,6){\circle*{0.15}}

\put(2.25, 6.25){\small $\bn{z}{n_j}$}

\put(7.15,5){\circle*{0.15}}

\put(7.15, 4.5){\small $h(\bn{z}{n_j})$}

\put(9,5.5){\circle*{0.15}}

\put(9.25, 5.5){\small $\bn{t}{n_j}$}




\put(3,5){\circle*{0.15}} \put(3.2, 4.5){\small $\bn{s}{n_j}$}

\put(4.5, 5){\circle*{0.15}}\put(4.25, 4.5){\small
$h^2(\bn{z}{n_j})$}

\put(4.5, 4){\vector(0,-1){2.25}}

\put(6.5,5){\circle*{0.15}}\put(5.9, 4.5){\small $\bn{u}{n_j}$}


\put(1.25,.25){\dashbox{.05}(7.15,6.5)}

\put(4.5 ,6.9){\small $V$}

\end{picture}
\caption{Case (1).} \label{fig:fourth}
\end{figure}

  In case (2) there is a point $\bn{k}{n_j}$ in between
  $\bn{u}{n_j}$ and $\bn{s}{n_j}$ such that $h^{-1}(\bn{k}{n_j})$
  is between $\bn{z}{n_j}$ and $h (\bn{z}{n_j})$.  Let $\ba w\in A$ be
  such that $\bn{k}{n_j}\to \ba w$.  Then we again have that $\ba
  z+\delta\le \ba w\le h(\ba z)-\delta$.  So $\ba w\nin\{ \ba z, h(\ba
  z)\}$.  Since $\ba z<\ba w< h(\ba z)$ we know that $h^{-1}(\ba
  z)<h^{-1}(\ba w)<\ba z$.  Let $\delta \ge \gamma>0$ be such that $h^{-1}(\ba w)< \ba
  z-\gamma$.  Since $h^{-1}(\bn{k}{n_j})$
  is between $\bn{z}{n_j}$ and $h (\bn{z}{n_j})$ and since
  $\bn{z}{n_j}\to \ba z>\ba z-\gamma$, it must be the case that
  $h^{- 1}(\bn{k}{n_j})\to h^{-1}(\ba w)\ge \ba z-\gamma$, a
  contradiction.

  Notice that in both case (1) and (2) we get a contradiction from
  the assumption that $\limsup \{A_n\}\ni \ba x$, a folding point, and
  the contradiction will follow if we are in case (1) or (2) for any
  $\delta'\le \delta$.  So if there is some $\delta'\le \delta$ such
  that (1) or (2) holds then we are finished.  So for case (3), suppose instead
  that for all $\delta'\le \delta$ we have $$h(\bn{z}{n_j})<h^2(\bn{z}{n_j})\le \bn{u}{n_j}$$
  for all $j\ge J'$ (where the points $\bn{u}{n_j}$ now depends
  on the choice of $\delta'$.)  Moreover, since $\bn{u}{n_j}\to
  h(\ba z)$ as $\delta'\to 0$ and $j\to \infty$.  Thus we have
  $h^2(\bn{z}{n_j})\to h(\ba z)$ and $h(\bn{z}{n_j})\to h(\ba z)$.  This
  implies that $h^2(\ba z)=h(\ba z)$, a contradiction since $h$ is a
  homeomorphism and $h(\ba z)\neq \ba z$.

  Thus, under the assumption that $\ba z\neq h(\ba z)$ we have shown
  that $\limsup \{A_n\}$ does not contain a folding point of $K$.

  Now suppose that $\ba z=h(\ba z)$.  Choose some point $\ba w$ such that $\ba z<\ba w\le h(\ba w)\le
  \ba x$ where $\ba x$ is the nearest folding point greater than $\ba z$.  Let $A_0$ be
  the arc in $C$ with endpoints $\ba z$ and $h(\ba w)$.  Notice that since $h$ is order-preserving, $\ba z <h^{-1}(\ba
  w)$.    Let $V$ be a neighborhood of $A_0$ and
  $m\in \N$ be chosen as in Lemma \ref{lemma5.1}.
  Suppose that $\limsup \{A_n\}\ni \ba x$.  Let $(n_j)_{j\in \N}$
  be chosen such that $\bn{t}{n_j}\in E_m$, $\bn{t}{n_j}\to \ba x$
  and $\bn{t}{n_j}\in A_{n_j}$.  Let $\delta>0$ be such that $\ba
  z<\ba z+\delta\le h(\ba z+\delta)<h^{-1}(\ba w-\delta)\le
  \ba w-\delta<\ba w$.

  Let $C_{n_j}$ be defined as before. Let $\bn{u}{n_j}$ be the point in
  $C_{n_j}$ with the same $m$th-coordinate as $h(\ba z+\delta)$, and
  let $\bn{s}{n_j}$ be the point in $C_{n,j}$ with the same $m$th
  -coordinate as $h^{-1}(\ba w-\delta)$.  Let $\gamma_{n_j}>0$ be
  defined so that $\bn{u}{n_j}+\gamma_{n_j}<\bn{s}{n_j}-\gamma_{n_j}$.  Let $\bn{v}{n_j}\in C_{n_j}$ such that
  $$\bn{u}{n_j}+\gamma_{n_j}<\bn{v}{n_j}<\bn{s}{n_j}-\gamma_{n_j}$$  Then
  $\bn{v}{n_j}\to \ba v\in A_0$ (or at least a subsequence of $\bn{v}{n_j}$) with $$h(\ba z+\delta)\le \ba v\le h^{-1}(\ba w-\delta)$$
  Since $\bn{z}{n_j}<\bn{v}{n_j}<h (\bn{z}{n_j})$, we have
  $h^{-1}(\bn{z}{n_j})<h^{-1}(\bn{v}{n_j})<\bn{z}{n_j}$.  Thus
  $h^{-1}(\bn{v}{n_j})$ converges to some point less than $\ba
  z+\delta$.  So $h^{-1}(\ba v)<\ba z+\delta$.  This is a
  contradiction to the fact that $h$ is order preserving.

\end{proof}

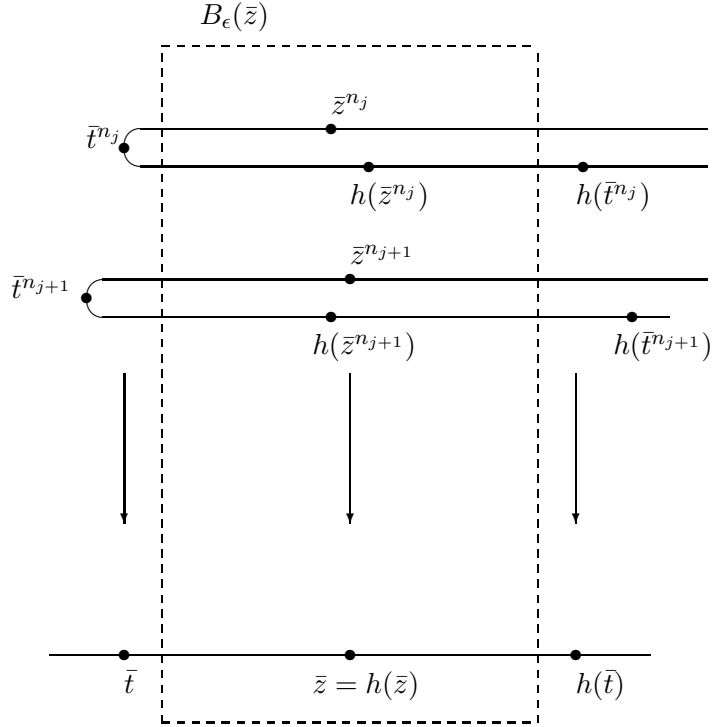
\begin{figure}
\unitlength=10mm
\begin{picture}(10,10)(0,0)


\put(1,1){\line(1,0){8}}

\put(2,1){\circle*{0.15}} \put(2,.5){\small $\ba t$}
\put(2,4.75){\vector(0,-1){2}}

\put(5,1){\circle*{0.15}} \put(4.5,.5){\small $\ba z=h(\ba z)$}
\put(5,4.75){\vector(0,-1){2}}

\put(2.5,.1){\dashbox{.1}(5,9)} \put(3,9.4){\small $B_\epsilon(\ba
z)$}

\put(8,1){\circle*{0.15}} \put(8,.5){\small $h(\ba
t)$}\put(8,4.75){\vector(0,-1){2}}


\put(2.25,8){\line(1,0){7.5}} \put(4.75,8){\circle*{0.15}}
\put(4.75,8.2){\small $\bn{z}{n_j}$}

\put(2.25,7.75){\oval(.5,.5)[l]}\put(2,7.75){\circle*{0.15}}
\put(1.5,7.75){\small $\bn{t}{n_j}$}

\put(2.25,7.5){\line(1,0){7.5}}\put(5.25,7.5){\circle*{0.15}}
\put(5, 7){\small $h(\bn{z}{n_j})$} \put(8.1, 7.5){\circle*{0.15}}
\put(8, 7){\small $h(\bn{t}{n_j})$}



\put(1.75,6){\line(1,0){8}} \put(5,6){\circle*{0.15}}
\put(5,6.2){\small $\bn{z}{n_{j+1}}$}

\put(1.75,5.75){\oval(.5,.5)[l]}\put(1.5,5.75){\circle*{0.15}}
\put(.5,5.75){\small $\bn{t}{n_{j+1}}$}

\put(1.75,5.5){\line(1,0){7.5}}\put(4.75,5.5){\circle*{0.15}}
\put(4.5, 5){\small $h(\bn{z}{n_{j+1}})$} \put(8.75,
5.5){\circle*{0.15}} \put(8.5, 5){\small $h(\bn{t}{n_{j+1}})$}

%
%
%
%
%
%
%
%
%
%
%
%
%
%
%
%
%
%
%
%
%
%
%
%
%
%
%
%
%
%
%
%
%
%
%
%
%
%

\end{picture}
\caption{Proof of Lemma~\ref{zfoldpoint}.} \label{fig:fold}
\end{figure}

\begin{lemma}\label{zfoldpoint}
  Let $\ba z\in F$.  Let $\bn{z}{n}\to \ba z$.  Let $A_n$ be the arc
  in $K$ with endpoints $\bn{z}{n}$ and $h(\bn{z}{n})$.  Then $A_n\to
  \{\ba z\}$ in the Hausdorff metric.
\end{lemma}

\begin{proof}
  Let $\epsilon>0$.  We will show that there is some $N\in \N$ such
  that if $n\ge N$ then $A_n\subseteq B_\epsilon(\ba z)$.  Notice that
  if there is some $m, N\in \N$ such that $A_n\cap E_m=\emptyset$
  for all $n\ge N$ then there is some $N'\ge N$ such that
  $A_n\subseteq B_\epsilon(\ba z)$ for all $n\ge N'$.  So instead
  suppose that for all $m\in \N$ and for all $N\in \N$ there is some
  $n\ge N$ such that $A_n\cap E_m\nempty$.  Fix $m\in \N$.  Suppose
  that there is some $N\in \N$ such that for all $n\ge N$, $A_n\cap
  E_m=\{\bn{t}{n}\}\subseteq B_\epsilon(\ba z)$.  Then for all $n\ge
  N$ we have $A_n\subseteq B_\epsilon(\ba z)$ and we are finished.
  So suppose we can pass to a subsequence, $(n_j)_{j\in \N}$ such
  that $\bn{t}{n_j}\nin B_\epsilon(\ba z)$ for all $j\in \N$.  By
  passing to another subsequence if necessary, let $\ba t$ be the
  limit of the $\bn{t}{n_j}$s.  Then we see that $\ba t$ is in the
  composant of $\ba z$, $C$, and $\ba t\neq \ba z$.  Without loss of
  generality, assume that $\ba t<\ba z$.  Then since $\ba z\in F$,
  $h(\ba z)=\ba z$ and so $\ba t\le h(\ba t)<\ba z$.  Assume that $\ba
  t, h(\ba t)\nin B_\epsilon(\ba z)$ (take $\epsilon$ smaller if
  necessary).  Let $J\in \N$ be large enough so that $\bn{z}{n_j},
  h(\bn{z}{n_j})\in B_{\epsilon/2}(\ba z)$ for all $j\ge J$.

  Since $\bn{t}{n_j}\in E_m$, $h(\bn{t}{n_j})$ is on an arc component
  of a link of $\C C_{p,n}$ that does not contain any other
  $p$-points, and hence no other $m$-points.  Since
  $\bn{z}{n_j}<\bn{t}{n_j}<h(\bn{z}{n_j})$ we have that $\bn{t}{n_j}<h (\bn{z}{n_j})<
  h(\bn{t}{n_j})$.  See Figure \ref{fig:fold}.  This implies that $h(\bn{t}{n_j})\to h(\ba t)$ and
  $\ba z=h(\ba z)\le h(\ba t)$, a contradiction. Thus we have shown that
  $A_n\to \{\ba z\}$ in the Hausdorff metric.
\end{proof}

Now define $$H:K\times[0,1]\to K$$ by the following.  For all $\ba
x\in F$ let $$H(\ba x, t)=\ba x=h(\ba x)$$  For $x\nin F$ we have
two cases: either $h(\ba x)=\ba x$ or $h(\ba x)\neq \ba x$.  If
$h(\ba x)=\ba x$ then define $$H(\ba x, t)=\ba x=h(\ba x)$$  If
instead we have $h(\ba x)\neq \ba x$, then let $A$ be the arc from
$\ba x$ to $h(\ba x)$, and let $V$ and $m$ be defined as in Lemma
\ref{lemma5.1}. Then let
$$H(\ba x, t)=\pi_m^{-1}|_{A}\left[(1-t)\pi_m(\ba x)+t(\pi_m(h(\ba x))\right]$$

\begin{theorem}
  The homeomorphism $h$ is isotopic to the identity via the isotopy $H$.
\end{theorem}

\begin{proof}
  First we need to show that $H:K\times [0,1]\to K$ as defined above is continuous.  Let $\ba z\in K$, $t\in [0,1]$, and let $(\bn{z}{n}, t_n)\to (\ba
  z, t)$.  Suppose that $\ba z\nin F$.  Then by Lemma
  \ref{znotfoldpoint} we know that since $\bn{z}{n}\to \ba z$ and
  $h(\bn{z}{n})\to h(\ba z)$, the arcs $A_n\to A$ in the Hausdorff
  metric.  So there is an $N\in \N$ such that for all $n\ge N$,
  $A_n\subseteq V$, the neighborhood of $A$ guaranteed by Lemma \ref{lemma5.1}.
   Hence we must have $H(\bn{z}{n},t_n)\to H(\ba z,
   t)$ because the $V$ and $m$ defining $H$ for $(\bn{z}{n}, t_n)$
   is the same as the $V$ and $m$ defining $H$ for $(\ba z, t)$.

  Next suppose that $\ba z\in F$.  Then by Lemma \ref{zfoldpoint}
 we see that the arcs $A_n$ converge in the Hausdorff metric to
  the point $\ba z$.  Thus $H(\bn{z}{n},t_n)\to H(\ba z, t)=\ba z$.
  Thus $H$ is continuous.

  To see that each $H$ is an isotopy, fix some $t\in [0,1]$ and consider the function $h_t$
  given by $h_t(\ba x)=H(\ba x, t)$.  Suppose that $h_t(\ba
  x)=h_t(\ba y)$ for some $\ba x, \ba y\in K$.  Then by definition
  $\ba x$ and $\ba y$ are on the same composant, $C$, of $K$.  If
  $h(\ba x)=\ba x$ then $h_t(\ba y)=h_t(\ba x)=\ba x$ by definition.
  If  $\ba y\neq \ba x$, say $\ba y<\ba x$, then we have $\ba y<\ba
  x=h(\ba x)<h(\ba y)$ which contradicts the fact that $h$ is
  order-preserving.  So suppose, without loss of generality, that $\ba x<h(\ba
  x)$ and $\ba y<h(\ba y)$ (if one maps larger and the other maps
  smaller we quickly get a contradiction).  Then let $\ba
  a=\min\{\ba x, \ba y\}$ and let $\ba b=\max\{h(\ba x), h(\ba
  y)\}$.  Let $A$ be the arc with endpoints $\ba a$ and $\ba b$.
  Let $m$ be chosen so that $\pi_m|_{A}$ is a homeomorphism.  Then
  it is a simple algebra exercise to see that we must have
  $\pi_m(\ba x)=\pi_m(\ba y)$.  Since $\pi_k|_A$ is a homeomorphism
  for all $k\ge m$ this implies that $\ba x=\ba y$  Hence $h_t$ is
  one-to-one.  It is similarly easy to see that $h_t$ is a
  surjection.  Thus $H$ is an isotopy from $h$ to the identity.
\end{proof}

Since $h=\sh^{-a}\circ h'$, we have shown the following:

\begin{corollary}
  Let $h':\ilim{T}\to \ilim{T}$ be a homeomorphism.  Then there is
  an integer $a$ such that $h'$ is isotopic to $\sh^a$.
\end{corollary}

\bibliographystyle{plain}
\bibliography{rainesbib}

\end{document}